\documentclass[11pt]{amsart}
\usepackage{amscd,amssymb,latexsym}
\usepackage{tikz}
\usepackage{tkz-graph}
\usepackage{pgf}
\usepackage{hyperref}
\usepackage{graphicx}
\usepackage{multicol}
\usepackage{enumitem}
\usetikzlibrary{arrows.meta,cd,calc}
\theoremstyle{plain}

\newtheorem{thm}{Theorem}[section]
\newtheorem{lem}[thm]{Lemma}
\newtheorem{prop}[thm]{Proposition}
\newtheorem{cor}[thm]{Corollary}

\newtheorem{asspt}[thm]{Assumption}
\newtheorem{obstruction}{Obstruction}
\theoremstyle{remark}
\newtheorem{rem}[thm]{Remark}

\theoremstyle{definition}
\newtheorem{dfn}[thm]{Definition}
\newtheorem{example}[thm]{Example}

\newcommand{\bc}{\mathbb{C}}

\newcommand{\bz}{\mathbb{Z}}

\newcommand{\PP}{\mathbb{P}}
\newcommand{\CC}{\mathbb{C}}

\newcommand{\bt}{\mathbb{T}}

\newcommand{\g}{\gamma}
\newcommand{\aaa}{\alpha}
\newcommand{\bbb}{\beta}

\DeclareMathOperator{\mat}{Mat}
\DeclareMathOperator{\Hom}{Hom}

\DeclareMathOperator{\rk}{Rank}
\DeclareMathOperator{\crk}{Corank}
\DeclareMathOperator{\lcm}{lcm}
\DeclareMathOperator{\fox}{Fox}

\DeclareMathOperator{\Char}{\operatorname{Char}}
\DeclareMathOperator{\ab}{\textbf{ab}}

\numberwithin{equation}{section}

\newcommand\enet[1]{\renewcommand\theenumi{#1}
\renewcommand\labelenumi{\theenumi}}

\title[Characteristic varieties of graph manifolds]
{Characteristic varieties of graph manifolds and quasi-projectivity of fundamental groups of algebraic links}

\author[E.~Artal]{Enrique Artal Bartolo}

\author[J.I.~Cogolludo]{Jos\'e Ignacio Cogolludo-Agust{\'\i}n}
\address{Departamento de Matem\'aticas, IUMA, Facultad de Ciencias\\
Universidad de Zaragoza\\
c/ Pedro Cerbuna 12\\
E-50009 Zaragoza SPAIN}
\email{artal@unizar.es,jicogo@unizar.es}

\author[D.~Matei]{Daniel Matei}

\address{Institute of Mathematics of the Romanian Academy\\
P.O. Box 1-764\\ RO-014700 Buch\-arest\\ Romania}
\email{Daniel.Matei@imar.ro
}

\dedicatory{To Papadima who has been a true inspiration in our research.}

\thanks{The first two named authors are partially supported by
MTM2016-76868-C2-2-P and Grupo ``\'Algebra y Geometr{\'\i}a'' of
Gobierno de Arag\'on/Fondo Social Europeo.
The third named author was partially supported by the Romanian Ministry of National Education,
CNCS-UEFISCDI, grant PNII-ID-PCE-2012-4-0156 and FMI 53/10 (Gobierno de Arag{\'o}n).}

\begin{document}

\begin{abstract}
\hspace*{-1pt} The present paper studies the structure of characteristic varieties of fundamental groups of
graph manifolds. As a consequence, a simple proof solving a question posed by Papadima on the 
characterization of algebraic links that have quasi-projective fundamental groups is provided. The type 
of quasi-projective obstructions used here are in the spirit of Papadima's original work.
\end{abstract}

\maketitle

In \cite{P07} \c{S}tefan Papadima studied the difference between global and local
fundamental groups. A \emph{global group} is the  fundamental group of a smooth quasi-projective variety
and it is also known in the literature as a \emph{quasi-projective group}.
A \emph{local group} is the fundamental group of a \emph{small} representative
of $X\setminus Y$, where $(Y,0)\subset (X,0)$ are germs of analytic isolated
singularities at $0\in\mathbb{C}^N$. The main question, as stated in the cited report~\cite{P07},
is to decide when a local group is global.

As a consequence of Zariski--Lefschetz theory (and its local version by Hamm--L{\^e}~\cite{hl:cras})
it is enough to restrict the attention to smooth quasi-projective surfaces (for global groups)
and complement of (eventually empty) curves in normal surface singularities. 
Papadima's question started with the \emph{simplest}
local groups, i.e., the fundamental groups of algebraic links in~$\mathbb{S}^3$.

There is a straightforward partial answer: consider a quasi-homogeneous singularity, that is, the set of 
zeroes $V(F)\subset \CC^N$ of a quasi-homogeneous polynomial $F$ in $N$ variables. Then the local group of the 
quasi-homogeneous singularity $(V(F),0)\subset (\CC^N,0)$ is a global group, since $V(F)$ is a \emph{small} 
representative and $\CC^N\setminus V(F)=\PP^N\setminus (V(F)\cup H)$, where $H$ is the hyperplane at infinity.
In \cite{P07}, Papadima proves that the local group is not global for \emph{almost} all algebraic links
of more that two irreducible components which cannot be realized by quasi-homogeneous equations.
The proof uses the particular structure of the characteristic varieties of quasi-projective groups. 
Later on, more complete answers have been provided. In~\cite{FS:14} Friedl and Suciu showed that prime 
components of closed $3$-manifolds (or link complements) with quasi-projective fundamental groups are 
graph manifolds~\cite{wal:67b}; Biswas and Mj~\cite{BM:15} characterized quasi-projective fundamental 
groups of compact $3$-manifolds. While these works answer Papadima's question we want to give an 
alternative proof using techniques similar to those originally used by Papadima, namely the use of quasi-projectivity
obstructions coming from characteristic varieties.

The goal of this paper is two-fold: first, a method to calculate fundamental groups and to compute characteristic 
varieties of graph $3$-manifolds is developed in section~\ref{sec-graph} and second, this method will be used in 
section~\ref{sec:main} to complete Papadima's proof of Theorem~\ref{thm:main}: \emph{the fundamental group of an 
algebraic link in $\mathbb{S}^3$ is global if and only it comes from a germ of plane
curve singularity having the topological type of a quasihomogeneous curve}.
In section~\ref{sec:settings} a general introduction is given on both algebraic links and characteristic varieties.
It is worth mentioning that the proof of the main theorem is constructive in the sense that it allows one to 
effectively detect the obstruction of a given local group to be global. In section~\ref{sec:examples} several 
examples are discussed in order to visualize the techniques used in the proof of the main theorem.
This in an example of how Papadima's work has influenced our trajectory.

The key object behind the proof in~\cite{P07} is a generalization of the Alexander polynomial known as 
characteristic varieties. Characteristic varieties are subvarieties of the complex torus $\mathbb{T}_G$ whose 
dimension is the rank~$b_1(G)$ of $G$. They provide a stratification of the space of characters of $G$.
They can be defined as the jump loci of the homology of $G$ with local coefficients
$$\Char_k(G):={\xi \in \mathbb{T}_G \mid \dim H_1(G;\underline{\CC}_\xi) \geq k}.$$

The idea of the proof in~\cite{P07} is based on the following fact~\cite{DPS4}:
\begin{obstruction}
\em Let $\Sigma_1,\Sigma_2$ two distinct irreducible components of $\Char_k(G)$, 
the $k$-th characteristic variety of a quasi-projective group $G$.
Then $\Sigma_1\cap\Sigma_2$ is finite. 
\end{obstruction}
If~$G_K$ is the fundamental group of an algebraic link with $r$~components ($b_1(G_K)=r$), 
the zero locus of the multi-variable Alexander polynomial $\Delta_K(t_1,\dots,t_r)\!\in\!\mathbb{Z}[t_1^{\pm 1},\dots,t_r^{\pm 1}]$
coincides with the codimension-$1$ part of $\Char_1(G)$; a formula for $\Delta_K$ is found in~\cite{en}.
If $r>1$ and $\Delta_K$ has more than one \emph{essential} variable, then there are non-disjoint hypersurfaces
which intersect at a codimension-$2$ subvariety, which is infinite if $r>2$. The aforementioned obstruction works only for the 
algebraic links~$K$ (with at least three components) such that the geometric decomposition of $(\mathbb{S}^3,K)$
has at least two Seifert pieces containing components of~$K$ (see disccussion in section~\ref{sec:charvar}). 

In this work, we will apply further obstructions based on~\cite{ACM-quasi-projective,DPS4}, to provide a complete answer
for complements of non-empty curves in normal surface singularities, which includes the case of algebraic links. 
We will make use of another obstruction (an immediate consequence of~\cite[Proposition~6.5(3)]{ACM-quasi-projective}):
\begin{obstruction}
\em Assume $G$ is a quasi-projective group, $\Sigma_1$ is an irreducible component of $\Char_k(G)$, 
and $\Sigma_2$ is an irreducible component of $\Char_{k+1}(G)$ such that $\Sigma_2\subsetneqq\Sigma_1$,
then $\Sigma_2$ is a point. 
\end{obstruction}
A direct application of this obstruction allows one to cover some algebraic links with two components
but it is not enough to give a complete answer. 
Nevertheless, we are going to prove in this work that the group of any algebraic link~$K$
(which is not topologically equivalent to a link given by quasihomogeneous equations)
admits a finite index subgroup $H$ which does not pass the above obstruction.
Since quasi-projectivity is inherited by finite-index subgroups~\cite{se:95,gr:57,gr:58},
the result will follow.

These techniques will apply for more general local groups.
In order to establish the result we will study the characteristic varieties of graph manifolds~\cite{wal:67b}.
The plumbing construction of W.~Neumann~\cite{neu:81} and the presentation of fundamental groups for these 
manifolds~\cite{mum:61,eh:98,Cohen-Suciu-boundary} will be particularly useful.

\section{Settings}
\label{sec:settings}

\subsection{Algebraic links}\label{sec-links}
\mbox{}

Let $f:(\mathbb{C}^2,0)\to(\mathbb{C},0)$ be a germ of an analytic function defining an isolated singularity.
By Milnor theory~\cite{mil}, there exists $\varepsilon_0>0$ for which $f$ is defined on the Euclidean ball
$\mathbb{B}^4_{\varepsilon_0}\subset\mathbb{C}^2$ (centered at~$0$) of radius $\varepsilon_0$ such that
$C:=f^{-1}(0)$ is transversal to $\mathbb{S}^3_{\varepsilon}$ $\forall\varepsilon\in(0,\varepsilon_0)$.
As a consequence $K_{\varepsilon}:=C\cap\mathbb{S}^3_{\varepsilon}$ is a compact oriented manifold of
dimension~$1$ and hence $(\mathbb{S}^3_{\varepsilon},K_{\varepsilon})$ is a link.

It is known that the topological type of the pair $(\mathbb{S}^3_{\varepsilon},K_{\varepsilon})$
does not depend on~$\varepsilon$ and it will be denoted as $(\mathbb{S}^3,K)$. It is a link with~$r$
connected components, where~$r$ is the number of irreducible factors of~$f\in\mathbb{C}\{x,y\}$.
Moreover, the pair $(\mathbb{B}^4_{\varepsilon_0},C)$ is homemorphic to the open cone over
$(\mathbb{S}^3,K)$. In particular, the local fundamental group associated to the complement of
the zero locus of~$f$ is isomorphic to $\pi_1(\mathbb{S}^3\setminus K)=:G_K$.

\begin{example}
Let $f(x,y):=x^\alpha y^\beta (x^m-y^n)$, where $\alpha,\beta\in\{0,1\}$ and $m,n\in\mathbb{N}$.
The singularity defined by~$f$ is quasihomogeneous and it defines an algebraic link
of $\alpha+\beta+d$ components, where $d:=\gcd(m,n)$. For this singularity the radius $\varepsilon$
can be chosen to be~$\infty$ and then $\pi_1(\mathbb{C}^2\setminus C)\cong G_K$, i.e., the local
group is global.
\end{example}
 
The main tool to study the local singularities is the embedded resolution, i.e.,
a proper holomorphic map $\pi:(X,D)\to(\mathbb{C}^2,0)$, where $X$ is a smooth surface,
$D=\pi^{-1}(0)$ and $\pi^{-1}(C)$ is a normal crossing divisor. There exists a unique
minimal embedded resolution. This information is usually encoded in the dual graph~$\Gamma$
of $\pi^{-1}(C)$. In this graph, each irreducible component of~$D$ is represented by
a vertex~$v$ and denoted $D_v$, each irreducible component of the strict transform of~$C$ is represented
by an arrowhead and each ordinary double point of $\pi^{-1}(C)$ is represented by an edge
whose extremities correspond to the irreducible components of $\pi^{-1}(C)$ containing the
branches. 

The vertices $v$ (not the arrowheads) of $\Gamma$ are weighted by $e_v:=D_v^2=(D_v, D_v)_X$ the 
self-intersection numbers of each divisor $D_v$. 
The incidence matrix $A$ of the dual graph of~$D$
(with the weights in the diagonal) is a negative definite matrix. The resolution is minimal
if and only if the ($-1$)-vertices of~$\Gamma$ have degree at least~$3$. The vertices admit
another weight, the multiplicities $m_{v,f}$. They can be defined as the valuation of $f\circ\pi$
at~$D_v$. We denote also by~$b_v$ the number of arrowheads neighboring~$v$. These numbers
are related by the following equality:

\begin{equation}\label{eq-euler-mult}
A\cdot (m_{v,f})_{v\in \Gamma}+(b_v)_{v\in \Gamma}=\mathbf 0
\end{equation}

For the multiplicities, if $f=f_1\cdot\ldots\cdot f_r$ is the decomposition in irreducible factors,
it is also useful to consider the tuple $m_v=(m_{v,f_1},\dots,m_{v,f_r})$ for the multiplicities of each branch
(satisfying $m_{v,f}=m_{v,f_1}+\dots+m_{v,f_r}$).

The dual graph also represents the plumbing graph of the algebraic link, see~\cite{neu:81}. Each vertex represents 
an oriented $\mathbb{S}^1$-fiber bundle over $\mathbb{S}^2$ with Euler number the self-intersection and 
each edge joining two vertices represents the gluing of the two corresponding fiber bundles after emptying solid 
tori and interchanging sections and fibers.
The arrows represent fibers in the attached vertices and form the link~$K$. 

\begin{example}\label{ex-acampo0}
Let us consider $f(x,y)=(y-x^2)(y^2-x^5)$.
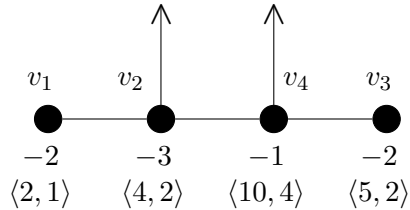
\begin{figure}[ht]
\begin{center}
\begin{tikzpicture}
\node[draw,circle,fill] (v0) at (-1.5,0) {};
\node[draw,circle,fill] (v1) at (0,0) {};
\node[draw,circle,fill] (v2) at (1.5,0) {};
\node[draw,circle,fill] (v3) at (3,0) {};
\node[rotate=90](flecha1) at (0.0025,1.4) {$\boldsymbol{>}$};
\node[rotate=90](flecha1) at (1.5025,1.4) {$\boldsymbol{>}$};
\draw (v0) -- (v1);
\draw (v1) -- (v2);
\draw (v2) -- (v3);
\draw (0,0) -- (0,1.5);
\draw (1.5,0) -- (1.5,1.5);
\node at (-1.6,.5) {$v_1$};
\node at (-1.6,-.5) {$-2$};
\node at (-1.6,-1) {$\langle 2,1\rangle$};
\node at (-.4,.5) {$v_2$};
\node at (-0.1,-.5) {$-3$};
\node at (-0.1,-1) {$\langle 4,2\rangle$};
\node at (2.9,.5) {$v_3$};
\node at (2.9,-.5) {$-2$};
\node at (2.9,-1) {$\langle 5,2\rangle$};
\node at (1.8,.5) {$v_4$};
\node at (1.4,-.5) {$-1$};
\node at (1.4,-1) {$\langle 10,4\rangle$};
\end{tikzpicture}
\caption{Dual Graph}
\label{fig:acampo0}
\end{center}
\end{figure}
The dual graph of the singularity is in Figure~\ref{fig:acampo0}; for each vertex the first weight is $e_i$ and the second
one is the pair of multiplicities for each branch (the first one for~$y^2-x^5$ and the second one for $y-x^2$).
Using the Seifert--van Kampen Theorem it will be proved in Example~\ref{ex-grupo-acampo0} that
\begin{equation}\label{pres:acampo0}
G_K=\langle
\g_{v_1},\g_{v_3}\mid [\g_{v_1}^2,\g_{v_3}^2]=1
\rangle.
\end{equation}
The generators $\g_{v}$ represent fibers for the fibration associated with the vertex~$v$.
\end{example}

\begin{dfn}
A vertex~$v\in \Gamma$ is a \emph{branching vertex} if its degree~$\rho(v)$ in the graph (including the arrows) 
is at least~$3$. A branching vertex is said to be \emph{essential} if it neighbors an arrow. 
\end{dfn}

\begin{prop}
An algebraic link can be realized by a quasihomogeneous equation if and only if there is at most one branching
vertex.
\end{prop}

\begin{proof}
The ($\Rightarrow$) part is easily seen. 

To prove the converse let us start with the simplest case of an algebraic knot. The topology of an
algebraic knot is defined by its Puiseux pairs, see e.g.~\cite{brk:86}. The number of Puiseux pairs determine 
the shape of the minimal graph of the embedded resolution (equivalently, the plumbing
graph of the link), see Figure~\ref{fig:Puiseux}.
\begin{figure}[ht]
\begin{center}
\begin{tikzpicture}[yscale=.5,vertice/.style={draw,circle,fill,minimum size=0.2cm,inner sep=0}]
\node[vertice] at (-3,0) {};
\node[vertice] at (0,0) {};
\node[vertice] at (0,-3) {};
\node[vertice] at (3,0) {};
\node[vertice] at (3,-3) {};
\node[vertice] at (6,0) {};
\node[vertice] at (6,-3) {};
\draw (-3,0)--(-2.5,0);
\draw[dashed] (-2.5,0)--(-.5,0);
\draw (-.5,0)--(0,0)--(0,-.75);
\draw (0,-2.25)--(0,-3);
\draw[dashed] (0,-.75)--(0,-2.25);
\draw (0,0)--(.5,0);
\draw[dashed] (.5,0)--(2.5,0);
\draw (2.5,0)--(3,0)--(3,-.75);
\draw (3,-2.25)--(3,-3);
\draw[dashed] (3,-.75)--(3,-2.25);
\draw (3,0)--(3.5,0);
\draw[dashed] (3.5,0)--(5.5,0);
\draw (5.5,0)--(6,0)--(6,-.75);
\draw (6,-2.25)--(6,-3);
\draw[dashed] (6,-.75)--(6,-2.25);
\draw[-{[scale=2]>}] (6,0)--(7,0);
\end{tikzpicture}
\caption{Dual graph of an irreducible germ of plane curve singularity}
\label{fig:Puiseux}
\end{center}
\end{figure}
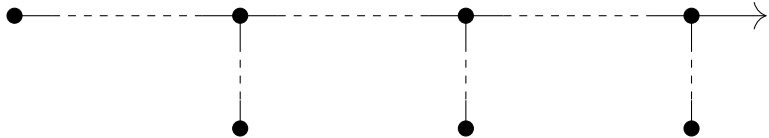
It is well known that the curve has exactly one Puiseux pair (i.e., one branching point) if and only if the singularity can 
be realized by a quasihomogeneous equation.

In the general reducible case, the graph without the arrowheads must be linear, and it may have arrows at the extremities, 
and at most one arrow attached to an inner vertex. If there is no inner arrow, there are two possibilities. 
For the first one, there is at most one arrow at each extremity; by minimality we should have arrows at the extremities 
and only one vertex, i.e., an ordinary double point which is quasihomogeneous. For the second one, there are two arrows at 
one of the extremities, and at most one at the other one. By minimality, the first extremity has self-intersection~$-1$ and 
the rest of them are of self-intersection~$-2$. It is easily seen that those singularities
are topologically equivalent to $\mathbb{A},\mathbb{D}$ singularities.
If there is an inner vertex, this one corresponds to one Puiseux pair singularity and the result follows.
\end{proof}

\begin{rem}
Only the trivial knot (smooth branch) and the Hopf link (ordinary double point) have no branching vertex.
\end{rem}

\begin{rem}
The algebraic links for which $G_K$ has been found to be non quasi-projective in~\cite{P07} are those with
at least three connected components, at least two branching vertices and one essential vertex.
\end{rem}

One of the main invariants of an algebraic link is its Alexander polynomial which can be computed using
A'Campo's formula for the zeta function of the monodromy.

\begin{prop}[\cite{ac:75}]
The Alexander polynomial $\Delta_K(t)$ equals
$$
(t-1)\prod_{v\in \Gamma} (t^{m_v}-1)^{\rho(v)-2}.
$$ 
\end{prop}

This formula has been generalized by Eisenbud and Neumann~\cite{en} for the multi-variable Alexander polynomial
which works also for integral homology spheres.

\begin{prop}[{\cite[Theorem~12.1]{en}}]\label{thm-alex-en}
The multi-variable Alexander polynomial of the link $K$ associated with the irreducible factorization of 
the reduced germ $f=f_1\cdots f_r$ is given as follows
$$
\Delta_K(\underline{t})=\prod_{v\in \Gamma} \left(\underline{t}^{m_{v}}-1\right)^{\rho(v)-2},
$$ 
where $\underline{t}=(t_1,\dots,t_r)$ and $\underline{t}^{m_{v}}=t_1^{m_{v,f_1}}\cdots t_r^{m_{v,f_r}}$.
\end{prop}

\subsection{Characteristic varieties and obstructions}
\label{sec:charvar}
\mbox{}

Characteristic varieties form a stratification of the space of characters of a group.
The special structure of characteristic varieties of quasi-projective groups provide a number of non-trivial 
obstructions for a group to be quasi-projective. Roughly speaking, the positive-dimensional part of characteristic 
varieties comes from maps from the quasi-projective variety onto \emph{orbicurves}; as an extension of 
Zariski--Lefschetz theory, we may say that most properties of characteristic varieties come from curves. 
these ideas appeared in Beauville's work~\cite{Be} for the projective case and were extended to the quasi-projective 
case by Arapura~\cite{ara:97}. Further properties were found by Budur, Dimca, Libgober, Suciu, and the authors in 
several papers, see~\cite{ACM-quasi-projective} for references. Let us introduce what will be needed later.

For a group $G$ (say, finitely presented) we consider its character torus 
\[
\mathbb{T}_G=H^1(G;\mathbb{C}^*)=\Hom(G;\mathbb{C}^*). 
\]
If $G/G'$ is isomorphic $\mathbb{Z}^r\oplus\bigoplus_{k=1}^s\mathbb{Z}/n_k$, then
$\mathbb{T}_G$ is a complex algebraic group isomorphic to $(\mathbb{C}^*)^r\times\prod_{k=1}^s \mu_{n_k}$,
where $\mu_n$ is the group of complex $n$-roots of unity.

Let $X$ be a topological space (having the homotopy type of a $CW$-complex) such that $\pi_1(X)\cong G$;
a character $\xi\in\mathbb{T}_G$ defines a local system of coefficients $\underline{\mathbb{C}}_\xi$ for which 
one may compute its cohomology $H^*(X;\underline{\mathbb{C}}_\xi)$. The first cohomology group depends 
only on $G$ and will be denoted by $H^1(G;\underline{\mathbb{C}}_\xi)$.

The $k$-th characteristic variety of $G$ is defined as:
$$
\Char_k(G):=\{\xi\in\mathbb{T}_G\mid \dim H^1(G;\underline{\mathbb{C}}_\xi)\geq k\}.
$$
Note that $\Char_k(G)\supset \Char_{k+1}(G)$ produces a stratification of the space of characters.
We may also consider each characteristic strata
$$
\mathcal{V}_k(G):=\{\xi\in\mathbb{T}_G\mid \dim H^1(G;\underline{\mathbb{C}}_\xi)= k\}=\Char_k(G)\setminus\Char_{k+1}(G).
$$

The first obstruction we can apply is the following one: the irreducible components
of $\Char_k(G)$ are subtori translated by torsion elements (see~\cite{ara:97}). The following are finer obstructions to being 
a global group (a quasi-projective group) which will be used in this paper. They can be deduced 
from~\cite[Proposition~6.5]{ACM-quasi-projective}.

\begin{obstruction}\label{thm-obs}
Let $G$ be a quasi-projective group. Let $A$ be a subset of $\mathcal{V}_k(G)$, $k>0$, and let
$B\subset\mathcal{V}_{k+\ell}$, $\ell>0$. If $B\subset\bar{A}$ then $B$ is finite. 
\end{obstruction}

Without loss of generality, one can assume that $A\subset V_1$ (resp. $B\subset V_2$) is contained in an irreducible 
component of $\Char_k(G)$ (resp. $\Char_{k+\ell}$) which is not a component of $\Char_{k+1}(G)$ (resp. $\Char_{k+\ell+1}$). 
By~\cite[Proposition~6.5(4)]{ACM-quasi-projective} $V_2$ is an irreducible component of $\Char_{k+1}(G)$ and hence 
$B\subset \bar{A}\subset V_1$ implies $B$ is finite by~\cite[Proposition~6.5(3)]{ACM-quasi-projective}.

For the sake of completeness, we will describe a classical method to explicitly compute 
$\dim H^1(G;\underline{\mathbb{C}}_\xi)$ from a presentation of~$G$. Assume that $G$ has a presentation
$$
G=
\left\langle
x_1,\dots,x_m\left|
R_1(\mathbf{x}),\dots,R_\ell(\mathbf{x})
\right.
\right\rangle.
$$
Let $\ab:G\to G/G'$ be the abelianization map and let $\Lambda:=\mathbb{C}[G/G']$; 
we denote by $t_j:=\ab(x_j)$, i.e., the ring
$\Lambda$ is the quotient of $\mathbb{C}[t_1^{\pm 1},\dots,t_m^{\pm 1}]$ by the ideal
generated by the abelianization of the relations $R_j(\mathbf{x})$, $j=1,\dots,\ell$. In an equivalent
way, $\mathbb{T}_G$ is a subgroup of a torus $(\mathbb{C}^*)^m$. The Fox matrix
$\fox(\mathbf{t})\in\mat(\ell\times m;\Lambda)$ is defined as follows:
its $(i,j)$-entry is $\frac{\partial R_i(\mathbf{x})}{\partial x_j}$ in~$x_1,\dots,x_m$. 
Let us recall the definition of the Fox derivatives (with respect to the abelianization morphism):
\begin{equation}
\label{eq:fox-derivatives}
\frac{\partial 1}{\partial x_j}=0,\quad
\frac{\partial x_i}{\partial x_j}=\delta_{i,j},\quad
\frac{\partial (w_1\cdot w_2)}{\partial x_j}=
\frac{\partial w_1}{\partial x_j}+\ab(w_1)\frac{\partial w_2}{\partial x_j}.
\end{equation}
Here the abelianization morphism $\ab:G\to G/G'$ is extended to the complex group algebra $\Lambda$ 
so that the element $\ab(w_1)\frac{\partial w_2}{\partial x_j}\in \Lambda$ is correctly defined.
A character $\xi\in\mathbb{T}_G$ defines a morphism 
$\Lambda\to\bc$ denoted also by~$\xi$. The evaluation by 
$\xi$ of the matrix 
$\fox(\mathbf{t})$ yields a matrix
$\fox(\xi)\in\mat(\ell\times m;\mathbb{C})$.

\begin{prop}\label{prop:corank}
Let $\xi\in\mathbb{T}_G$. Then,
$$
\dim_{\mathbb{C}} H^1(G;\underline{\mathbb{C}}_{\xi})=
\begin{cases}
m-\rk(\fox(\xi))-1&\text{ if }\xi\neq\mathbf{1},\\
m-\rk(\fox(\mathbf{1}))&\text{ if }\xi=\mathbf{1}.
\end{cases}
$$ 
\end{prop}

\begin{proof}
Let us consider the $CW$-complex associated with the presentation
of~$G$ and let $p$ be the unique $0$-cell. 
The complex chain is given by
\begin{equation*} 
\begin{tikzcd}[row sep=0]
0\ar[r]&[-15pt]
C_2(X)=\mathbb{C}\langle R_1,\dots,R_\ell\rangle\ar[r,"\partial_2"]&[-3pt]
C_1(X)=\mathbb{C}\langle x_1,\dots,x_m\rangle
\ar[r,"\partial_1"]
&[-3pt]
C_0(X)=\mathbb{C}\langle p\rangle\ar[r]&[-16pt]0
\end{tikzcd}
\end{equation*}
where $\partial_1(x_j)=(\xi(x_j)-1)p$ and the matrix of $\partial_2$
is $^t\fox(\xi)$. The result follows.
\end{proof}

\begin{rem}
\label{rem:corank}
By Proposition~\ref{prop:corank} $\Char_k(G_K)\subset \{\mathbf{1}\}$ if $k\geq m$.
\end{rem}

The hypersurface part of the first characteristic variety $\Char_1(G_K)$
is the zero locus of the multi-variable Alexander polynomial $\Delta_K(\underline{t})$. In particular,
the following result can be proved using Obstruction~\ref{thm-obs}. Note that from Proposition~\ref{thm-alex-en}
the codimension~$1$ part is a union of tori translated by torsion elements.

\begin{prop}[\cite{DPS4}]
If $G$ is a global group with $r := b_1 (G)\geq 3$, then $\Delta(\underline{t})$ has a single essential variable, 
that is $\Delta(\underline{t})=P(u)$, with $P \in\mathbb{Z}[u^{\pm 1}]$ for some $u=\underline{t}^e$.
\end{prop}

Using Proposition~\ref{thm-alex-en} and \eqref{eq-euler-mult} it is possible to prove \cite[Theorem~4]{P07}, i.e.,
the algebraic links for which $G_K$ is non quasi-projective 
are those for which~$K$ has at least three connected components, at least two branching vertices, and one essential vertex.
We can extend these arguments for links with two components, for which we need the following result.

\begin{obstruction}[{\cite[Proposition~6.9]{ACM-quasi-projective}}]\label{prop-obs2}
Let $G$ be a quasi-projective group, and let $V_1\subset\Char_k(G)$, $V_2\subset\Char_\ell(G)$ be two distinct
irreducible components. If $\xi\in V_1 \cap V_2$, then this torsion point
satisfies $\xi \in\Char_{k+\ell}(G)$.
\end{obstruction}

Note that the multiplicities in $\Delta_K(t)$ are not completely related with
the hypersurfaces in $\Char_k(G_K)$, $k>1$. The following example deals with this issue.

\begin{example}
We continue with Example~\ref{ex-acampo0}. Note that
$$
\Delta_K(t_1,t_2)=(t_1^2 t_2+1)(t_1^5 t_2^2+1),\quad
\Delta_K(t)=(t^3+1)(t^7 +1)(t-1).
$$ 
The intersection of the two hypersurfaces
in the zero locus of $\Delta_K$ is $(t_1,t_2)=(-1,-1)$. This point
should be in $\Char_2(G_k)$ by Obstruction~\ref{prop-obs2}.
However, by Remark~\ref{rem:corank} we have $(-1,-1)\not\in\Char_2(G_K)$.
\end{example}

\section{Graph manifolds}\label{sec-graph}

Graph manifolds were classified by Waldhausen~\cite{wal:67b}. In this section we will focus on a special 
class of graph manifolds, namely oriented graph manifolds, made by oriented $\mathbb{S}^1$-bundles, which 
are best described by plumbing graphs~\cite{neu:81} as we did in~\S\ref{sec-links}. 

The main goal of this section is to provide a presentation of the fundamental group of the complement
of links whose connected components are fibers of some $\mathbb{S}^1$-bundles in such graph manifolds.
We follow the procedures explained in works of Mumford~\cite{mum:61} and Hironaka~\cite{eh:98}, see
also the recent work of Koll{\'a}r and N{\'e}methi~\cite{kol-nem:13}. A complete description is needed for the 
computation of the Fox matrix and characteristic varieties.

\subsection{The fundamental group of a plumbing graph}
\mbox{}

Let $\Gamma$ be a connected graph without loops. Let us decompose its set of vertices as a disjoint union $V\coprod H$,
where $V$ is non-empty and its elements are called \emph{vertices} and $H$ is a set of terminal vertices
(i.e., of degree 1) and called \emph{arrowheads}. The set of edges is also decomposed as a disjoint union $E\coprod F$; 
the elements of $E$ are called \emph{edges} and connect two vertices and those of $F$ are called \emph{arrows} and 
connect an arrowhead and a vertex. Note that $r:=|H|=|F|\geq 0$. Let us denote $n:=|V|>0$. For $v,w\in V$, $v\neq w$, 
we set $E_{v,w}$ as the set of edges connecting $v$ and~$w$.
Given an arrowhead $h$ we will denote by $v_h$ the unique vertex connected with $h$; the connecting arrow
will be called~$e_h$. For a vertex $v$ the set of vertices (resp. arrowheads) connected with $v$ is denoted by $V_v$ 
(resp.~$H_v$).

The vertices are weighted by two functions $g,\varepsilon:V\to\mathbb{Z}$; the function $\varepsilon$ is the 
\emph{Euler number} and the function $g$ is the genus, $g(v)\geq 0$, $\forall v\in V$; we will denote
$g(\Gamma):=\sum_{v\in V} g(v)$. 

This graph codifies a graph manifold $M$ with a link $L$ of $r$~components. Recall we are only considering oriented 
graph manifolds such that all its Seifert pieces are Seifert bundles over oriented surfaces.

The following two incidence matrices $A=(a_{v w})_{v,w\in V}$ and $B=(b_{vh})_{v\in V,h\in H}$
are considered:

\begin{equation}
\label{eq:AB}
a_{v w}=
\begin{cases}
\varepsilon(v)&\text{ if }v=w\\
|E_{v,w}|&\text{ otherwise },
\end{cases}
\quad \quad \quad
b_{v h}=\delta_{vv_h}=
\begin{cases}
1&\text{ if }h\in H_v\\
0&\text{ otherwise.}
\end{cases}
\end{equation}

For a presentation of $\pi_1(M\setminus L)$ we start by fixing a maximal tree~$T$ of $\Gamma$.
It is useful to fix a \emph{directed rooted} tree structure on $T$ away from the root. This can be 
extended first to an order on $V$ and then to one on $H$ such that given $h\in H_v$ and $h'\in H_w$, then 
$h<h'$ if $v<w$. Since $\Gamma$ contains no loops, this ordering produces an orientation respecting it.
Also, a linear order will be chosen for the arrowheads $H_v$ at each vertex~$v$.
Finally, each set of edges $E_{v,w}$ between $v,w$ will be ordered in such a way that the minimal element
is the unique edge $e_0\in T\cap E_{v,w}$ in the maximal tree $T$ that was fixed above. The set $E_v$ of
edges from $v$ inherits an order respecting both the order in $V_v$ and in $E_{v,w}$,
that is, if $e,e'\in E_{v,w}$, $e'\in E_{v,w'}$, then 
$$e<e' \Leftrightarrow \begin{cases}w<w'\\ \textrm{or} \\ w=w' \textrm{ and } e<e' \textrm{ as edges in }E_{v,w}. \end{cases}$$
The notation $E^*$ and $E^*_{v,w}$ refers to the respective edges outside of the maximal tree $T$.

The generator system for $\pi_1(M\setminus L)$ is obtained as follows:
\begin{enumerate}
\enet{\rm(G\arabic{enumi})}
\item\label{G1} 
$\gamma_v$ for $v\in V$.
\item \label{G2} 
$\gamma_h$ for $h\in H$.
\item \label{G3} 
$\gamma_e$ where $e\in E^*$.
\item\label{G4} 
$\aaa_{j,v},\bbb_{j,v}$, for $1\leq j\leq g(v)$ and $v\in V$.
\end{enumerate}

In order to describe the relations we introduce some notation. For each oriented edge $\vec{e}\in E$ starting at 
$v_1(e)$ and ending at $v_2(e)$ we denote:
\begin{equation}
\label{eq:gammae}
\gamma_{\vec{e}}:=
\begin{cases}
1&\text{ if }e\in T;\\
\gamma_e&\text{ if } e\notin T\text{ and }v_1(e)<v_2(e)
\\
\gamma_e^{-1}&\text{ if } e\notin T\text{ and }v_2(e)<v_1(e)
\end{cases}
\end{equation}
We have the following relations:
\begin{enumerate}
\enet{\rm(R\arabic{enumi})}
\item \label{r1}
$\displaystyle\left(\prod_{\vec{e}\in E_{v}} 
\gamma_{\vec{e}}\cdot\gamma_{v_2(e)}\cdot\gamma_{\vec{e}}^{-1}\right)\cdot
\left(\prod_{h\in H_v} \gamma_h \right)\cdot \gamma_v^{\varepsilon(v)}=\prod_{j=1}^{g(v)}[\aaa_{j,v},\bbb_{j,v}]$,
for $v\in V$.
\item \label{r2}
$[\gamma_h,\gamma_{v_h}]=1$ for $h\in H$.
\item\label{r3} 
$[\gamma_{v_1(e)},\gamma_{\vec{e}}\cdot\gamma_{v_2(e)}\cdot\gamma_{\vec{e}}^{-1}]=1$ for $e\in E^*$.
\item\label{r4} 
$[\gamma_v,\aaa_{j,v}]=[\gamma_v,\bbb_{j,v}]=1$ for $1\leq j\leq g(v)$ and $v\in V$.
\end{enumerate}

\begin{rem}
Note that the product in~\ref{r1} is taken respecting the order of the vertices in $V_v$. Moreover, for a fixed
$w\in V_v$ the product in ${\vec{e}\in \vec{E}_{v,w}}$ is also taken respecting the order in~$E_{v,w}$.

Also note that if $H\neq\emptyset$, then one relation in in~\ref{r2} is redundant. More precisely, any relation 
in~\ref{r2} is a consequence of the remaining relations in~\ref{r2} together with \ref{r1}+\ref{r3}+\ref{r4}. 
For symmetry reasons, we will consider all the relations for $h\in H$, but this fact will be systematically used.
\end{rem}

\begin{thm}
\label{thm:pi1}
The fundamental group of $M\setminus L$ is generated by \ref{G1}-\ref{G4} with relations \ref{r1}-\ref{r4}. 

As a consequence, this group admits a presentation with as many
generators as relations. If $r>0$, one of the relations in \ref{r2}
can be omitted.
\end{thm}

\begin{proof}
We divide the proof in two cases, depending on $L$:

\begin{enumerate}
 \item Case $L\neq\emptyset$.
Assume first $\Gamma$ is a tree. This particular case is proved by induction on the number $n$ of vertices. 
If $n=1$, we have a fiber bundle and the presentation in~\cite{s:33} is given by the generators
in~\ref{G1}, \ref{G2} and~\ref{G4}, with the relations in \ref{r1}, \ref{r2}
and~\ref{r4} (one of the relations in \ref{r2} being redundant).
To finish the induction argument, assume that~$n>1$. Let $v\in V$ a vertex of degree 1 in $V\subset \Gamma$. 
Since $V$ has at least two degree-one vertices, one can assume there is at least one arrowhead not in $H_v$.
We cut $\Gamma$ along this edge $e=\langle v,w\rangle$ and split $\Gamma$ in two graphs: 
a one-vertex graph $\Gamma'$ with an extra arrow $h_0\in H'_v$ coming from the edge $e$ and another graph $\Gamma''$
containing $\Gamma\setminus \{v\}$ with the extra arrow $h_1\in H''_w$.
Note that~$\Gamma''$ has at least two arrows. Let $m',m''$ be the number of generators of
the fundamental groups of the corresponding graph manifolds. In particular,
we have $m'-1,m''-1$ relations; the discarded relations corresponding to 
$h_0$ and an arrowhead of $\Gamma''$ different from~$h_1$.

The fundamental group of the graph manifold of~$\Gamma$ is obtained using the Seifert--van Kampen Theorem. Since 
two couples of generators are identified, the number generators is $m'+m''-2$; it equals the number
of relations but we can discard the commutation associated to~$h_1$, since it coincides to the one associated 
to~$h_0$ which is redundant. Hence we obtain a presentation with generators \ref{G1}, \ref{G2} and~\ref{G4}, 
using the relations in \ref{r1}, \ref{r2} and~\ref{r4} (one relation of \ref{r2} being redundant).

For the general case (still $L\neq\emptyset$), we proceed by induction on the first Betti number~$m$ of~$\Gamma$. 
The case $m=0$ has been treated. It is not hard to see, using the generalized Seifert--van Kampen Theorem
that we add in this case the generators in~\ref{G3} and the relations in~\ref{r3}.

\item Case $L=\emptyset$. The closed case follows the proof above. When the graph is split, then $\Gamma'$ and
$\Gamma''$ contain only one arrowhead. Using again the Seifert--van Kampen Theorem, we kill one generator without adding 
any relations, obtaining the result.
\end{enumerate}
\end{proof}

\begin{rem}
In fact, any edge and any arrow provides a commutator, but only the ones in~\ref{r2} and~\ref{r3} are needed. 
\end{rem}

The abelianization of this fundamental group is as follows. The proof
is straightforward.

\begin{lem}
The space $H_1(M\setminus L)$ is determined as follows:
\begin{enumerate}[label=\rm(\arabic{enumi})]
\item Its rank is $\crk (A)+r+2g(\Gamma)+b_1(\Gamma)$.
\item Its torsion can be obtained from the Smith form of $(A\vert B)\in M(n\times(n+r);\mathbb{Z})$.
\end{enumerate}
\end{lem}

Then the torus $\mathbb{T}:=H^1(M\setminus L;\mathbb{C}^*)=\Hom(H_1(M\setminus L;\mathbb{Z}),\mathbb{C}^*)=
\Hom(\pi_1(M\setminus L),\mathbb{C}^*)$ will be seen as a subtorus of $(\mathbb{C}^*)^{n+r+2g(\Gamma)+b_1(\Gamma)}$.
The coordinates of this torus will be $t_v$, $v\in V$, $t_h$, $h\in H$, $t_{i,v},s_{i,v}$, $i=1,\dots,g(v)$,
$v\in V$ and $t_e$, $e\in E^*$. The equations of $\mathbb{T}$ in these coordinates are
determined by the matrices $A$ and $B$ as follows:
\begin{equation}
\label{eq:prod}
\prod_{w\in V} t_w^{a_{v w}}\cdot\prod_{h\in H_v} t_h=1,\quad\forall v\in V.
\end{equation}

\begin{example}\label{ex-grupo-acampo0}
We illustrate this method with the link provided in Example~\ref{ex-acampo0}. The tree $T$ is rooted at $v_4$ and the 
order is reversed with respect to the labels, that is, $v_i<v_j$ if $j<i$. We find six generators; the generators 
$\gamma_{i}:=\gamma_{v_i}$, $1\leq i\leq 4$, come from~\ref{G1} while the generators $\delta_j:=\gamma_{h_j}$, $j=2,4$, 
come from~\ref{G2}. We have the following relations:
\begin{itemize}
\item[\ref{r1}]
$\gamma_2=\gamma_1^2$, $\gamma_1\gamma_4\delta_2=\gamma_2^3$, $\gamma_4=\gamma_3^2$
and $\gamma_2\gamma_3\delta_4=\gamma_4$.
\item[\ref{r2}] $[\gamma_2,\delta_2]=1$ and  $[\gamma_4,\delta_2]=1$.
\end{itemize}
Note that one of the relations~\ref{r2}, say the last one, is a consequence of the other ones. 
Using the relations~\ref{r1}, we can eliminate four generators, only $\gamma_1,\gamma_3$ are kept.
The last relation looks like:
$$
1=[\gamma_2,\delta_2]=[\gamma_1^2,\gamma_4^{-1}\gamma_1^{-1}\gamma_2^3]=[\gamma_1^2,\gamma_3^{-2}\gamma_1^5]
\Longleftrightarrow [\gamma_1^2,\gamma_3^2]=1,
$$
i.e., the presentation~\eqref{pres:acampo0}.

\end{example}

\begin{example}\label{ejm:notorsion} 
Let us compute now the fundamental group of the closed graph manifold in Figure~\ref{fig:closed}.
This manifold $M$ is also a torus bundle over~$\mathbb{S}^1$ with monodromy
$$
\mathcal M=\begin{pmatrix}
-1 & 1 \\
-5 & 4
\end{pmatrix}.
$$

\begin{figure}[ht]
\vspace{-4mm}
\begin{tikzpicture}
\node[fill,circle] (v0) at (-1,0) {};
\node[fill,circle] (v1) at (1,0) {};
\draw[style={bend left}] (v0) to (v1);
\draw[style={bend right}] (v0) to (v1);
\node[left] at (v0.west) {$-5$};
\node[above] at (v0.north) {$v_1$};
\node[right] at (v1.east) {$-1$};
\node[above] at (v1.north) {$v_2$};
\end{tikzpicture}
\caption{A closed graph manifold}
\label{fig:closed}
\end{figure}
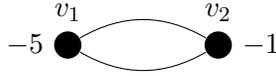

We choose the order so that $v_1<v_2$ and denote the edges in $E_{v_1,v_2}$ as $e_0, e_1$ where $e_0\in T$.
One obtains two generators $\gamma_1,\gamma_2$ of type \ref{G1} and a generator $\delta=\gamma_{e_1}$ of type~\ref{G3}. 
According to Theorem~\ref{thm:pi1} there are two relations of type~\ref{r1}:
\begin{enumerate}
\enet{($\text{rel}_{\text{\arabic{enumi}}}$)} 
\item $\gamma_2 \delta\gamma_2\delta^{-1}=\gamma_1^5$ (for $v_1\in V$);
\item\label{rel2} $\gamma_1 \delta^{-1}\gamma_1\delta=\gamma_2$ (for $v_2\in V$);
\end{enumerate}
no relations of type~\ref{r2} and one relation~\ref{r3}:
\begin{enumerate}
\enet{($\text{rel}_{\text{\arabic{enumi}}}$)} 
\setcounter{enumi}{2}
\item $[\gamma_1,\delta\gamma_2\delta^{-1}]=1$ for $e_1\in E^*_{v_1,v_2}$.
\end{enumerate}

One can use~\ref{rel2} to eliminate the generator $\gamma_2$ obtaining:
$$
\pi_1(M)=
\langle
\gamma,\delta\mid
[\delta\gamma\delta^{-1},\gamma]=1,\quad
(\delta^{-1}\gamma\delta) (\delta\gamma \delta^{-1})=\gamma^3
\rangle.
$$
The abelianization of the group $\pi_1(M)$ is isomorphic to $\mathbb{Z}$ (generated by the image of~$\delta$, 
and the class of $\gamma$ is trivial), and hence $\Lambda=\mathbb{C}[t^{\pm 1}]$.
The Fox matrix is:
$$
\begin{pmatrix}
0& 0\\
t-3+t^{-1}&0
\end{pmatrix}.
$$
Then $\Char_1(G)\subset\mathbb{C}^*$ is the set of zeroes of $t^2-3 t+1$ (which is also the characteristic 
polynomial of $\mathcal M$). Since the characteristic variety contains non-torsion isolated points, the group is non 
quasi-projective.
\end{example}

\subsection{Fox matrices of plumbing graphs}
\label{sec:fox}
\mbox{}

The twisted cohomology of $M\setminus L$ in degree~$1$ can be computed by means of the Fox matrix as 
described in section~\ref{sec:charvar}. Recall in this case it is a square matrix of size $n+r+b_1(\Gamma)+2g(\Gamma)$
%
which decomposes as:
\begin{equation}
\label{eq:fox}
\fox(\mathbf{t}):=\begin{pmatrix}
A(\mathbf{t})&\vert& B(\mathbf{t})&\vert&C(\mathbf{t})&\vert&A_g(\mathbf{t})\\
\hrulefill&&\hrulefill&&\hrulefill&&\hrulefill\\
\tilde{B}(\mathbf{t})&\vert& H(\mathbf{t})&\vert&0_{r\times b_1(\Gamma)}&\vert&0_{r\times 2 g(\Gamma)}\\
\hrulefill&&\hrulefill&&\hrulefill&&\hrulefill\\
\tilde{C}(\mathbf{t})&\vert&0_{b_1(\Gamma)\times r}&\vert&E(\mathbf{t})&\vert&0_{b_1(\Gamma)\times 2 g(\Gamma)}\\
\hrulefill&&\hrulefill&&\hrulefill&&\hrulefill\\
{\tilde{A}_g(\mathbf{t})}&\vert&0_{2 g(\Gamma)\times r}&\vert&
0_{2 g(\Gamma)\times b_1(\Gamma)}&\vert&G(\mathbf{t})
\end{pmatrix} 
\end{equation}
In what follows we will describe the submatrices of the Fox matrix.

\subsubsection{\texorpdfstring{The matrix $A(\mathbf{t})$}{The matrix A(t)}}
\mbox{}

This is a square matrix of order $n$. For $v\in V$, the $v$-row corresponds to the coefficients in the 
generators~\ref{G1} of the Fox derivatives~\eqref{eq:fox-derivatives} of the relations of type~\ref{r1} relation
\begin{equation}\label{eq-rel-v}
\left(\prod_{\vec{e}\in \vec{E}_{v}} \gamma_{\vec{e}}\cdot\gamma_{v_2(e)}\cdot\gamma_{\vec{e}}^{-1}\right)\cdot
\left(\prod_{h\in H_v} \gamma_h\right)\cdot \gamma_v^{\varepsilon(v)}=\prod_{j=1}^{g(v)}[\aaa_{j,v},\bbb_{j,v}].
\end{equation}
Hence, the entries of this matrix are as follows. 
Using~\eqref{eq:prod} note that the diagonal term $a_{v v}(\mathbf{t})$ is given by:
$$
\prod_{w\in V} t_w^{a_{v w}}\cdot\prod_{h\in H_v} t_h\cdot
\frac{1-t_v^{-\varepsilon(v)}}{t_v-1}=
\frac{1-t_v^{-\varepsilon(v)}}{t_v-1}.
$$
For $w\neq v$, the term $a_{v w}(\mathbf{t})$ is given by
$$
\sum_{\vec{e}\in E_{v,w}} T_{v,e} t_{\vec{e}},\quad \textrm{ where } \quad
T_{v,e}:=\prod_{e>e'\in E_{v}} t_{v_2(e')},
$$
and $t_{\vec{e}}$ follows the same conventions as $\gamma_{\vec{e}}$ ---\,see~\eqref{eq:gammae}.
Note that $A(\mathbf{1})=A$ as defined in~\eqref{eq:AB}.

\subsubsection{\texorpdfstring{The matrix $B(\mathbf{t})$}{The matrix B(t)}}
\mbox{}

It is a matrix of order $n\times r$ (it exists only if $r>0$). As for the previous matrices,
the $v$-row corresponds to the Fox derivation of the relation~\eqref{eq-rel-v} with respect to $\gamma_h$. 
The term $b_{v h}(\mathbf{t})$ equals
$$
T_v\prod_{h>h'\in H_v} t_{h'},\quad \textrm{ where } \quad 
T_v:=\prod_{\vec{e}\in E_{v}} t_{v_2(e)}.
$$
Note that $B(\mathbf{1})=B$ as defined in~\eqref{eq:AB}.

\subsubsection{\texorpdfstring{The matrix $C(\mathbf{t})$}{The matrix C(t)}}
\mbox{}

It is a matrix of order $n\times b_1(\Gamma)$. As for the previous submatrices $A(\mathbf{t})$ and $B(\mathbf{t})$,
the $v$-row corresponds to the Fox derivation of the relation~\eqref{eq-rel-v}, in this case with respect to $\gamma_e$
for $e\in E^*$. The term $c_{ve}(\mathbf{t})$ is obtained as
$$
c_{ve}(\mathbf{t})=
\begin{cases}
T_{v,e}(1-t_{v_2(e)})&\text{ if }\vec{e}\in E_v, v_2(e)>v,\\
T_{v,e}(t_{v_2(e)}-1)t_{\vec{e}}&\text{ if }\vec{e}\in E_v, v_2(e)<v,\\
0&\text{ otherwise}.
\end{cases}
$$
So, any column of  $C(\mathbf{t})$ has only two non-zero entries.

\subsubsection{\texorpdfstring{The matrices $A_g(\mathbf{t})$}{The matrix Ag(t)}}
\label{sec:Ag}
\mbox{}

This matrix is of order $n\times 2g(\Gamma)$. This can be seen as a block matrix whose blocks for any given $v\in V$ 
are $n\times 2g(v)$ matrices $A_{g,v}(\mathbf{t})$ whose columns are parametrized by $\aaa_{j,v},\bbb_{j,v}$, 
$j\in \{1,\dots,g(v)\}$. For any $A_{g,v}(\mathbf{t})$ only the $\aaa_{j,v}$ and $\bbb_{j,v}$ columns are non-zero.
The elements on these columns are given by:
$$
(s_{j,v}-1,-(t_{j,v}-1)).
$$

\subsubsection{\texorpdfstring{The matrix $\tilde{B}(\mathbf{t})$}{The matrix \tilde{B}(t)}}
\mbox{}

It is a matrix of order $r\times n$ corresponding to the Fox derivation of relations of type~\ref{r2} 
with respect to the variables $\gamma_v$, $v\in V$.
The term $\tilde{b}_{hv}(\mathbf{t})$ is obtained as
$$
\tilde{b}_{hv}(\mathbf{t})=
\begin{cases}
1-t_h&\textrm{ if }v_h=v\\
0&\textrm{ otherwise. }
\end{cases}
$$
Note that each row has exactly one non-zero element.
\subsubsection{\texorpdfstring{The matrix $H(\mathbf{t})$}{The matrix H(t)}}
\mbox{}
This is a diagonal matrix of order $r\times r$ corresponding to the Fox derivation of relations of type~\ref{r2} 
with respect to the arrowhead generators $\gamma_h$, $h\in H$.
The diagonal entry corresponding to~$h$ is $(t_{v_h}-1)$.

\subsubsection{\texorpdfstring{The matrix $\tilde{C}(\mathbf{t})$}{The matrix \tilde{C}(t)}}
\mbox{}

It is a matrix of order $b_1(\Gamma)\times n$ corresponding to the Fox derivation of relations of type~\ref{r3} 
with respect to the generators $\gamma_v$, $v\in V$. In order to describe it, let $e$ be an edge in $E^*$ oriented 
in the standard form so that $v_1(e)<v_2(e)$. Then
$$
\tilde{c}_{ev}(\mathbf{t})=
\begin{cases}
1-t_{v_2(e)}&\textrm{ if }v=v_1(e),\\
t_{e}(t_{v_1(e)}-1)&\textrm{ if }v=v_2(e),\\
0&\textrm{ otherwise}.
\end{cases}
$$

\subsubsection{\texorpdfstring{The matrix $E(\mathbf{t})$}{The matrix E(t)}}
\mbox{}

It is the diagonal matrix of order $b_1(\Gamma)\times b_1(\Gamma)$ corresponding to the Fox derivation of 
the relations of type~\ref{r3} with respect to the generators $\gamma_e$, $e\in E^*$. The diagonal entry for $e$ is 
$(t_{v_1(e)}-1)(1-t_{v_2(e)})$ where $v_1(e)<v_2(e)$.

\subsubsection{\texorpdfstring{The matrix $\tilde{A}_g(\mathbf{t})$}{The matrix \tilde{A}g(t)}}
\label{sec:tildeAg}
\mbox{}

This is the matrix of order $2g(\Gamma)\times n$ corresponding to the Fox derivatives of the relations of 
type~\ref{r4} with respect to the generators $\gamma_w$, $w\in V$. Symmetrically as was the case for $A_g(\mathbf{t})$,
it can be seen as a block matrix whose blocks are $2g(w)\times n$ matrices $\tilde{A}_{g,w}(\mathbf{t})$. All entries in 
their columns are zero except at the $w$ row where the respective entries are
$$
(1-t_{j,w},1-s_{j,w}).
$$

\subsubsection{\texorpdfstring{The matrix $G(\mathbf{t})$}{The matrix G(t)}}
\mbox{}

This is an order $2g(\Gamma)\times 2g(\Gamma)$ matrix. One can think of it as a block matrix whose blocks
$G_{v,w}(\mathbf{t})$ are given by $2g(v)\times 2g(w)$ matrices for any pair of vertices $v,w\in V$. One has 
the following description of the blocks:
$$
G_{v,w}(\mathbf{t})=
\begin{cases}
(t_v-1) 1_{2 g(v)}&\textrm{ if }v=w\\
0_{2g(v)\times 2g(w)}&\textrm{ otherwise. }\\
\end{cases}
$$

\subsection{Properties of Fox matrices of plumbing graphs}
\mbox{}

With the description of the Fox Matrix $\fox(\mathbf{t})$ given in~\ref{eq:fox}, one can in principle compute the 
characteristic varieties for any pair $(M,L)$ ---\,as far as computing power allows. 

To provide some computation-free results, in this section we will calculate the (first) local system cohomology
corresponding to some particular characters.
These results will be used to apply obstruction theorems to quasi-projectivity such as Obstruction~\ref{thm-obs}. 
Let us consider the following sets of characters
$$
\mathcal{B}_0:=\left\{\xi\in\mathbb{T} \left| 
\array{ll}
\xi(t_h)=1 & \forall h\in H,\\
\xi(t_v)=1 & \forall v\in V,\\
\left(\xi(t_{j,v}),\xi(s_{j,v})\right)_{j=1,...,g(v)}\neq \mathbf{1} & \textrm{if } g(v)>0
\endarray
\right.
\right\},
$$
$$
\mathcal{B}_1:=\{\xi\in \mathcal{B}_0\mid \xi(t_e)=1, \forall e\in E^*\}.
$$

\begin{lem}
The following properties hold:
\begin{multicols}{3}
\begin{enumerate}[label=\rm(\alph{enumi})]
\item $A(\xi)=A$ if $\xi\in \mathcal{B}_1$.
\item $B(\xi)=B$ if $\xi\in \mathcal{B}_0$.
\item $C(\xi)=0$ if $\xi\in \mathcal{B}_0$.
\item $\tilde{B}(\xi)=0$ if $\xi\in \mathcal{B}_0$.
\item $H(\xi)=0$ if $\xi\in \mathcal{B}_0$.
\item $\tilde{C}(\xi)=0$ if $\xi\in \mathcal{B}_0$.
\item $E(\xi)=0$ if $\xi\in \mathcal{B}_0$.
\item $G(\xi)=0$ if $\xi\in \mathcal{B}_0$.
\end{enumerate}
\end{multicols}
Moreover, if $\xi\in \mathcal{B}_0$ and $g(v)>0$, then 
\begin{enumerate}[label=\rm(\roman{enumi})]
\item the only non-zero $v$-row of $A_g(\xi)$ is in the block $A_{g,v}(\xi)$.
\item the only non-zero $v$-column of $\tilde{A}_g(\xi)$ is in the block $\tilde{A}_{g,v}(\xi)$.
\end{enumerate}
\end{lem}
The proof of this lemma follows immediately from the description of these submatrices as given 
in~\ref{sec:Ag} and \ref{sec:tildeAg}.
We may use it to apply elementary operations to $\fox(\xi)$ in order to compute its rank.

\begin{lem}
If $g(v)>0$ for some $v\in V$, then the following properties hold:
\begin{enumerate}
\enet{\rm(\arabic{enumi})} 
\item If $\xi\in \mathcal{B}_0$, then $A_{g,v}(\xi)$ can be transformed into a matrix with only one non-zero element 
(in the $v$-row) using only column operations.
\item If $\xi\in \mathcal{B}_0$, then $\tilde{A}_{g,v}(\xi)$ can be transformed into a matrix with only one non-zero 
element (in the $v$-column) using only row operations.
\item If $\xi\in \mathcal{B}_0$, then $v$-column and the $v$-row of the matrix $A(\xi)$ can be transformed into zero,
using both row and column operations.
\end{enumerate}
\end{lem}

\begin{asspt}\label{as-nd}
We will assume that $A$ is a negative definite matrix. This is the case when the graph manifold comes from singularity theory. 
\end{asspt}

Let us fix $\varepsilon>0$ such that the matrix $A(\mathbf{t})$ is negative definite for any value of $\mathbf{t}$ such that 
$t_v=1$ and $|t_e-1|\leq\varepsilon$ for any $v\in V$ and $e\in E^*$. 
Such a value for $\varepsilon$ exists, since $A(\mathbf{1})=A$. Consider
$$
\mathcal{B}_\varepsilon:=\{\xi\in \mathcal{B}_0\mid |\xi(t_e)-1|\leq\varepsilon\}.
$$
Note that $\mathcal{B}_1\subset \mathcal{B}_\varepsilon\subset \mathcal{B}_0.$

\begin{thm}\label{thm:corank1}
If $\xi\in \mathcal{B}_\varepsilon$,
then 
$$
\crk(\fox(\xi))=\sum_{v\in V}\max\{0,2 g(v)-1\}+b_1(\Gamma)+r.
$$ 
\end{thm}

\begin{proof}
After a reordering the matrix becomes
\begin{equation*}
\fox(\mathbf{\xi}):=\begin{pmatrix}
A^{g=0}(\xi)&\vert& 0&\vert&0&\vert&0\ \\
\hrulefill&&\hrulefill&&\hrulefill&&\hrulefill\\
0&\vert& 0&\vert&1^{g\neq 0}&\vert&0\ \\
\hrulefill&&\hrulefill&&\hrulefill&&\hrulefill\\
0&\vert&1^{g\neq 0}&\vert&0&\vert&0\ \\
\hrulefill&&\hrulefill&&\hrulefill&&\hrulefill\\
0&\vert&0&\vert&0&\vert&0\ 
\end{pmatrix} 
\end{equation*}
The matrix $A^{g=0}(\xi)$ is a negative definite square matrix whose size is $\#\{v\in V\mid g(v)=0\}$ and
$1^{g\neq 0}$ is the identity matrix with size $\#\{v\in V\mid g(v)>0\}$, i.e., the rank of $\fox(\mathbf{\xi})$ is
\[
\#\{v\in V\mid g(v)=0\}+2\#\{v\in V\mid g(v)>0\}=n+\#\{v\in V\mid g(v)>0\}.
\]
Since the size of $\fox(\mathbf{\xi})$ is $n+r+b_1(\Gamma)+2g(\Gamma)$,
the result follows.
\end{proof}

Assume now that $g(w)>0$ for some $w\in V$. Let us consider $\mathcal B_{\varepsilon,w}\subset \mathbb T$ defined
as $\mathcal{B}_\varepsilon$, with the only exception that $\xi(t_{j,w})=\xi(s_{j,w})=1$, for $1\leq j\leq g(w)$.
The proof of this result is analogous to that of Theorem~\ref{thm:corank1}.

\begin{thm}\label{thm:corankw}
If $\xi\in B_{\varepsilon,w}$ 
then 
$$
\crk(\fox(\xi))=2 g(w)+\sum_{w\neq v\in V}\max\{0,2 g(v)-1\}+b_1(\Gamma)+r.
$$  
\end{thm}

\begin{cor}\label{cor-nqp}
If 
either $\Gamma$ has two vertices with positive genus or it has one such a vertex
and $b_1(\Gamma)>0$ then $G$ is not quasi-projective.
\end{cor}

\begin{proof} 
From the definitions note that $\mathcal{B}_{\varepsilon,w}\subset\overline{\mathcal{B}}_\varepsilon$.
Let us assume that $w\in V$ is a vertex of positive genus, in particular $\mathcal{B}_{\varepsilon,w}\neq \emptyset$
and thus infinite. Let
$\xi_w\in\mathcal{B}_{\varepsilon,w}$ and let $\xi\in\mathcal{B}_\varepsilon$. Note that $\xi,\xi_w\neq\mathbf{1}$.
By Proposition~\ref{prop:corank} and Theorems \ref{thm:corank1} and~\ref{thm:corankw} one obtains
\begin{gather*}
\dim_\mathbb{C} H^1(M\setminus L;\underline{\mathbb{\bc}}_\xi)=k
<k+1=
\dim_\mathbb{C} H^1(M\setminus L;\underline{\mathbb{\bc}}_{\xi_w})
\end{gather*}
where
\[
k=\sum_{v\in V}\max\{0,2 g(v)-1\}+b_1(\Gamma)+r-1> 0
\]
by hypothesis. In particular $\mathcal{B}_\varepsilon\subset\mathcal{V}_k$, $k>0$, 
$\mathcal{B}_{\varepsilon,w}\subset\mathcal{V}_{k+1}$ and both sets are infinite, hence by Obstruction~\ref{thm-obs}, 
the group $\pi_1(M\setminus L)$ is not quasi-projective.
\end{proof}

\section{Applications to algebraic links: main theorem}
\label{sec:main}
Assume now $(M,L)=(\mathbb{S}^3,K)$ where $K$ is an algebraic link. The dual graph~$K$ is a tree,
the incidence matrix~$A$ is negative definite and the genus function vanishes. Our purpose is to 
give a result on quasi-projectivity of such groups, that is, on whether or not such local groups 
can be global, i.e., fundamental groups of quasi-projective varieties. The key tool here will be to 
use Corollary~\ref{cor-nqp}. 

\begin{thm}[Papadima~\cite{P07}, Friedl--Suciu~\cite{FS:14}, Biswas--Mj~\cite{BM:15}]
\label{thm:main}
The local group of an algebraic link is quasi-projective if and only if it is the group of a quasihomogeneous singularity.
\end{thm}

\begin{proof}
Since quasi-projectivity is inherited by finite-index subgroups, we will seek for finite-index
subgroups satisfying the hypotheses of Corollary~\ref{cor-nqp}.

In order to do so consider $L$ an algebraic link associated with $(V(f),0)\subset (\CC^2,0)$ and $\Gamma$ its 
associated graph. Let~$e:=\lcm\{m_{v,f} \mid v\in V\}$ be the least common multiple of the multiplicities 
$m_{v,f}$ at all vertices in~$\Gamma$ and $(M,\tilde L_e)$ the $e$-fold cyclic cover of $(\mathbb{S}^3,L)$ 
ramified along all the components of $L$. The result is a consequence of the following Lemma.
\end{proof}

\begin{lem}
\label{lemma:cubierta}
Assume $L$ is an algebraic link which is not of quasihomogeneous type. Then $(M,\tilde L_e)$ as defined above 
is a graph manifold satisfying the hypothesis of Corollary{\rm~\ref{cor-nqp}}.
\end{lem}

\begin{proof}
The way to construct the graph of $(M,\tilde L_e)$ is as follows. Let $f(x,y)=0$ be the equation
of a local germ singularity yielding~$L$. Let $X_\varepsilon\subset (\CC^2,0)$ be a closed Milnor ball and consider
the following diagram, where $\sigma:Y_\varepsilon\to X_\varepsilon$ is the minimal embedded resolution,
$\tilde{f}:=f\circ\sigma$, $\tilde{Y}_\varepsilon$ is the pull-back, $\hat{Y}_\varepsilon$ is its
normalization and $Z_\varepsilon$ is its resolution:
\[
\begin{tikzcd}
Z_\varepsilon\ar[r]&\hat{Y}_\varepsilon\ar[r]&\tilde{Y}_\varepsilon\ar[r]\ar[d]&Y_\varepsilon\ar[d,"\tilde{f}"]\\
&&\mathbb{C}\ar[r,"t\mapsto t^e"]&\mathbb{C}
\end{tikzcd}
\]
Note that the manifold $M$ can be seen as the boundary of either one of the manifolds 
$\tilde{Y}_\varepsilon$, $\hat{Y}_\varepsilon$ or $Z_\varepsilon$.
Also, the preimage of $0$ in $Z_\varepsilon$ is a normal crossing divisor, whose dual graph is the plumbing graph of~$M$.
The arrows determining $\tilde{L}_e$ correspond to the preimage of the strict transform of $f^{-1}(0)$.

Note that $\hat{Y}_\varepsilon$ is not a complex analytic manifold since it admits quotient singularities
on the preimages of the double points of $\tilde{f}^{-1}(0)$; the preimage of $\tilde{f}^{-1}(0)$
is a $\mathbb{Q}$-normal crossing divisor, see~\cite{amo:jos}, admitting also a dual graph.
One can obtain the graph of~$M$ from this dual graph by replacing some edges by linear subgraphs,
corresponding to the Jung--Hirzebruch resolution of the quotient singularities. In particular,
the possible positive genus components of the divisor in $Z_\varepsilon$ appear already in $\hat{Y}_\varepsilon$.

Let $D\cong\mathbb{P}^1$ be a branching component of $\tilde{f}^{-1}(0)$ with multiplicity~$m$. It has $r$ neighbors 
with multiplicities $m'_1,\dots,m'_r$ ($r\geq 3$). Let $d:=\gcd(m,m'_1,\dots,m'_r)$. Then,
the preimage of $D$ in $\hat{Y}_\varepsilon$ is the union of $d$ connected components, and each one of them 
is an $\frac{m}{d}$-fold cyclic cover ramified at $r$ points with ramification index~$\frac{m'_i}{d}\mod\frac{m}{d}$.
These points will be actual ramification points only if $\frac{m'_i}{d}\not\equiv 0\mod\frac{m}{d}$.

In fact, one does not need to finish constructing the embedded resolution to study this cover. 
Let $\sigma_D':X_D'\to\mathbb{B}_\varepsilon$ be the composition of blowing-ups such that there is $P\in X_D'$ 
for which $D$ is the strict transform of the exceptional component of the blowing-up of~$P$. The total transform 
$(f\circ\sigma_D')^*(0)$ is a non-reduced curve (multiple components come from the exceptional divisor); its germ 
at~$P$ has~$r$ tangent directions with multiplicities~$m_1,\dots,m_r$. We have that $m_1+\dots+m_r=m$ and 
$\frac{m'_i}{d}\equiv\frac{m_i}{d}\mod\frac{m}{d}$.

Hence, these cyclic covers are actually ramified over more than two points. Such a cover is a positive genus surface.
Since we have at least two branch points, the statement holds.
\end{proof}

\section{Examples}
\label{sec:examples}
The purpose of this section is to discus some characteristic examples to visualize the different phenomena 
described in the proof of the main theorem.

\begin{example}
Let us consider the singularity provided in Example~\ref{ex-acampo0}. In this case the $\lcm$ of the multiplicities is $42$. 
After performing a 6-fold cyclic cover the following graph is obtained (see left-hand side of Figure~\ref{fig:6fold}),
which can be simplified without changing the graph manifold (see right-hand side of Figure~\ref{fig:6fold}).
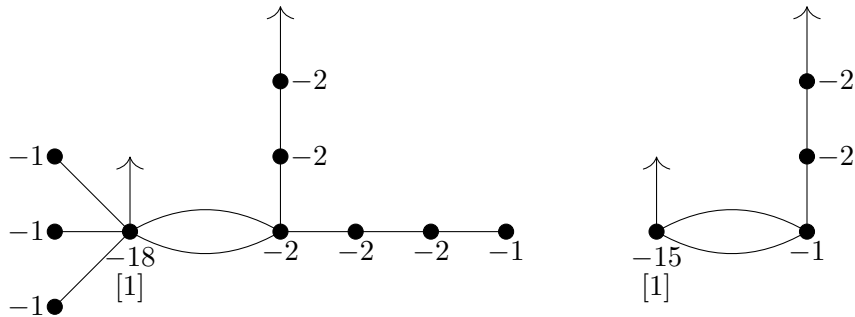
\begin{figure}[ht]
\begin{center}
\begin{tikzpicture}[vertice/.style={draw,circle,fill,minimum size=0.2cm,inner sep=0}]
\coordinate (v1) at (1,0);
\coordinate (v0) at ($-1*(v1)$);
\coordinate (u) at (0,1);
\node[vertice] at (v0) {};
\node[vertice] at (v1) {};
\draw[style={bend left}] (v0) to (v1);
\draw[style={bend right}] (v0) to (v1);
\node[below=2pt] at (v0) {$-18$};
\node[below=12pt] at (v0) {$[1]$};
\node[below] at (v1) {$-2$};
\node[vertice]  at ($2*(v1)$) {};
\node[below]  at ($2*(v1)$) {$-2$};
\node[vertice]  at ($3*(v1)$) {};
\node[below]  at ($3*(v1)$) {$-2$};
\node[vertice]  at ($4*(v1)$) {};
\node[below]  at ($4*(v1)$) {$-1$};
\draw (v1)--($4*(v1)$);

\node[vertice] at ($(v1)+(u)$) {};
\node[right] at ($(v1)+(u)$) {$-2$};
\node[vertice] at ($(v1)+2*(u)$) {};
\node[right] at ($(v1)+2*(u)$) {$-2$};
\draw[-{[scale=2]>}] (v1)--($(v1)+3*(u)$);
\draw[-{[scale=2]>}] (v0)--($(v0)+(u)$);

\node[vertice] at ($2*(v0)+(u)$) {};
\node[left] at ($2*(v0)+(u)$) {$-1$};
\draw (v0)--($2*(v0)+(u)$);
\node[vertice] at ($2*(v0)$) {};
\node[left] at ($2*(v0)$) {$-1$};
\draw (v0)--($2*(v0)$);
\node[vertice] at ($2*(v0)-(u)$) {};
\node[left] at ($2*(v0)-(u)$) {$-1$};
\draw (v0)--($2*(v0)-(u)$);

\begin{scope}[xshift=7cm]
\coordinate (v1) at (1,0);
\coordinate (v0) at ($-1*(v1)$);
\coordinate (u) at (0,1);
\node[vertice] at (v0) {};
\node[vertice] at (v1) {};
\draw[style={bend left}] (v0) to (v1);
\draw[style={bend right}] (v0) to (v1);
\node[below=2pt] at (v0) {$-15$};
\node[below=12pt] at (v0) {$[1]$};
\node[below] at (v1) {$-1$};

\node[vertice] at ($(v1)+(u)$) {};
\node[right] at ($(v1)+(u)$) {$-2$};
\node[vertice] at ($(v1)+2*(u)$) {};
\node[right] at ($(v1)+2*(u)$) {$-2$};
\draw[-{[scale=2]>}] (v1)--($(v1)+3*(u)$);
\draw[-{[scale=2]>}] (v0)--($(v0)+(u)$);
\end{scope}
\end{tikzpicture}
\caption{$6$-fold cyclic cover}
\label{fig:6fold}
\end{center}
\end{figure}
From this 6-fold cyclic cover it can already be deduced that the group is not quasi-projective, 
since its graph contains a cycle and a branched vertex.
\end{example}

\subsection{Characteristic varieties and monodromy for generalized A'Campo's links}
\mbox{}

Our purpose is to generalize the algebraic link presented in Example~\ref{ex-acampo0} 
in order to compare the behavior of characteristic varieties of their fundamental groups and the monodromy of their 
corresponding singularity. Let $f(x,y):=(y^q+x^p)(y^s+x^r)$ where $\gcd(p,q)=\gcd(r,s)=1$
and $\frac{p}{q}<\frac{r}{s}$. Since these singularities have long dual graphs for the resolution, we can use 
Eisenbud--Neumann results~\cite{en} to give a presentation of their fundamental groups:
\[
G:=\langle
\mu_x,\mu_y,\mu_z\mid
[\mu_z,\mu_x^s]=1,\mu_x^{s(r q-p s)}\mu_y^p\mu_z^{a r-b s}=1
\rangle
\]
where $b q-a p=1$. Let us denote $\alpha:=r q-p s$, $\beta=a r-b s$;
we choose $a,b\geq 0$ such that $\beta\geq 0$. The group algebra of $G/G'$ over the complex numbers
is the quotient $\bc[t^{\pm 1}_x,t^{\pm 1}_y,t^{\pm 1}_z]/(t_x^{s\alpha} t_y^p t_z^{\beta}-1)$. The Fox matrix is
\[
\begin{pmatrix}
(t_z-1)\dfrac{1-t_x^s}{1-t_x}&0&1-t_x^s\\
&&\\
\dfrac{1-t_x^{s\alpha}}{1-t_x}&t_x^{s\alpha}\dfrac{1-t_y^p}{1-t_y}&
t_x^{s\alpha} t_y^p \dfrac{1-t_z^\beta}{1-t_z}
\end{pmatrix}.
\]
The Fitting ideal generated by the $2$-minors of this matrix is 
\[
\dfrac{1-t_x^s}{1-t_x}
\left\langle
(t_z-1)\dfrac{1-t_y^p}{1-t_y},(1-t_x)\dfrac{1-t_y^p}{1-t_y},1-t_x^{s\alpha} t_z^\beta
\right\rangle.
\]
We obtain the following sets of irreducible components:
\begin{itemize}
 \item For each $\zeta_s\neq 1$, with $\zeta_s^s=1$ there is a component $\{t_x=\zeta_s,t_y^p t_z^{\beta}=1\}$.
 \item For each $\zeta_p\neq 1$, with $\zeta_p^p=1$ there is a component $\{t_y=\zeta_p,t_x^{s\alpha} t_z^{\beta}=1\}$.
\end{itemize}

The intersections of these components are 
\[
t_x=\zeta_s,\quad t_y=\zeta_p,\quad t_z^{\beta}=1.
\]
Let us consider now the Fitting ideals generated by $1$-minors.
Their elements are as above, excluding the cases when $t_z=1$.

\begin{example}
Let us consider $f_{a,b}(x,y)=(y-x^2)^a (y^2-x^5)^b$, where $\gcd(a,b)=1$.
Let $M$ be the complement of the algebraic link. The covering defining this monodromy
is determined by $\pi_1(M)\to\bz$ where $x\mapsto 2a+b$ and $y\mapsto 5a+2b$.
The monodromy is semisimple if and only if either $a$ or $b$ are even. Hence, the subtorus
defined by the covering passes through the intersection of two components of $\Char_1(G)$
if and only if the monodromy is not semisimple.
\end{example}

\begin{example}
Let us consider $f_{a,b}(x,y)=(y-x^3)^a (y^2-x^7)^b$, where $\gcd(a,b)=1$.
The reduced singularity is semisimple; in general, the monodromy is not semisimple
if and only if $a$ is even and $b\equiv 0\bmod 3$.
Let $M$ be the complement of the algebraic link, whose dual graph is shown in Figure~\ref{fig:1327}. 
The covering defining this monodromy is determined by $\pi_1(M)\to\bz$ where 
$x\mapsto 2a+b$ and $y\mapsto 7a+3b$. Hence, the subtorus
defined by the covering passes through the intersection of two components of $\Char_1(G)$
if and only if the monodromy is not semisimple.

\begin{figure}[ht]
\begin{center}
\begin{tikzpicture}[xscale=1.5]
\node[draw,circle,fill] (v) at (-3,0) {};
\node[draw,circle,fill] (v0) at (-1.5,0) {};
\node[draw,circle,fill] (v1) at (0,0) {};
\node[draw,circle,fill] (v2) at (1.5,0) {};
\node[draw,circle,fill] (v3) at (3,0) {};
\node[rotate=90](flecha1) at (0.0025,1.4) {$\boldsymbol{>}$};
\node[rotate=90](flecha1) at (1.5025,1.4) {$\boldsymbol{>}$};
\draw (v) -- (v3);
\draw (0,0) -- (0,1.5);
\draw (1.5,0) -- (1.5,1.5);
\node at (-3.1,.5) {$v_1$};
\node at (-3.1,-.5) {$-2$};
\node at (-3.1,-1) {$\langle a+2b\rangle$};
\node at (-1.6,.5) {$v_2$};
\node at (-1.6,-.5) {$-2$};
\node at (-1.6,-1) {$\langle 2(a+2b)\rangle$};
\node at (-.4,.5) {$v_3$};
\node at (-0.1,-.5) {$-3$};
\node at (-0.1,-1) {$\langle 3(a+2b)\rangle$};
\node at (2.9,.5) {$v_4$};
\node at (2.9,-.5) {$-2$};
\node at (2.9,-1) {$\langle 3a+7b\rangle$};
\node at (1.8,.5) {$v_5$};
\node at (1.4,-.5) {$-1$};
\node at (1.4,-1) {$\langle 2(3a+7b)\rangle$};
\end{tikzpicture}
\caption{Dual graph of $f_{a,b}$}
\label{fig:1327}
\end{center}
\end{figure}
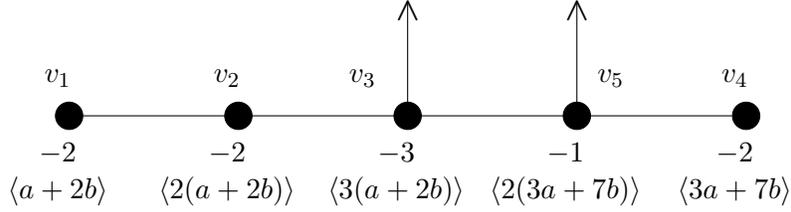

\end{example}

\begin{example}
As discussed in Example~\ref{ejm:notorsion}, the irreducible components of the characteristic variety of a general 
graph manifold could be subtori translated by \emph{non-torsion} elements. The following graph manifold provides 
an example whose characteristic varieties are not even translated tori (see Figure~\ref{fig:nosubtori}).

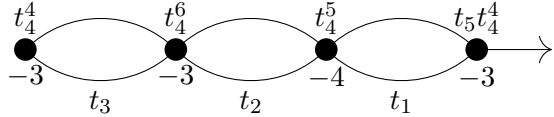
\begin{figure}[ht]
\begin{center}
\begin{tikzpicture}[]
\coordinate (v) at (-4,0);
\fill[black] (v) circle [radius=.15];
\node[below=2pt] at (v) {$-3$};
\node[above=2pt] at (v) {$t_4^4$};

\coordinate (v0) at (-2,0);
\fill[black] (v0) circle [radius=.15];
\node[below=2pt] at (v0) {$-3$};
\node[above=2pt] at (v0) {$t_4^6$};

\coordinate (v1) at (0,0);
\fill[black] (v1) circle [radius=.15];
\node[below=3pt] at (v1) {$-4$};
\node[above=2pt] at (v1) {$t_4^5$};

\coordinate (v2) at (2,0);
\fill[black] (v2) circle [radius=.15];
\node[below=3pt] at (v2) {$-3$};
\node[above=2pt] at (v2) {$t_5 t_4^4$};

\draw (v) to[out=45, in=135] (v0);
\draw (v)  to [out=-45, in=-135] node[pos=.5,below] {$t_3$} (v0) ;
\draw (v0) to[out=45, in=135]  (v1);
\draw (v0) to[out=-45, in=-135] node[pos=.5,below] {$t_2$} (v1);
\draw (v1) to[out=45, in=135] (v2);
\draw (v1) to[out=-45, in=-135] node[pos=.5,below] {$t_1$} (v2);

\draw[-{[scale=2]>}] (v2)--($(v2)+(1,0)$);
\end{tikzpicture}
\caption{Graph link admitting a no subtori component}
\label{fig:nosubtori}
\end{center}
\end{figure}

The torus $\bt_G$ is the maximal spectrum of $\CC[\bz^4\times\bz/2]$. Let us denote by $t_1,t_2,t_3,t_4$
the free components and by $t_5$ the torsion component satisfying $t_5^2=1$. The characteristic varieties
$\Char_1(G)$ and $\Char_2(G)$ have torsion translated subtori as components (of dimension~$3$). 
The variety $\Char_3(G)$ has three components. One of them is the subtorus $t_5=t_4=1$; the other 
ones are defined by $t_4=1$, $t_5=-1$ and $t_3^2 - 7 t_3 + 1=0$, i.e., they are non-torsion translated subtori.
Finally, $\Char_4(G)$ has three 1-dimensional components. Two of them are non-torsion translated subtori:
$t_3^2 - 7 t_3 + 1= t_4 - 1= t_2 + 1= t_5 - 1=0$. The third one, defined by
\[
t_3^2 t_2 + \frac{3}{4} t_3 t_2^2 - \frac{11}{2} t_3 t_2 +\frac{3}{4} t_3 + t_2= t_4 - 1= t_1 + 1= t_5 - 1=0
\]
is not a translated subtorus.
\end{example}

\bibliographystyle{amsplain}

\begin{thebibliography}{10}

\bibitem{ac:75}
N.~A'Campo, \emph{La fonction zeta d'une monodromie}, Comment. Math. Helv.
  \textbf{50} (1975), 233--248.

\bibitem{ara:97}
D.~Arapura, \emph{Geometry of cohomology support loci for local systems. {I}},
  J. Algebraic Geom. \textbf{6} (1997), no.~3, 563--597.

\bibitem{ACM-quasi-projective}
E.~Artal, J.I. Cogolludo-Agust{\'i}n, and D.~Matei, \emph{Characteristic
  varieties of quasi-projective manifolds and orbifolds}, Geom. Topol.
  \textbf{17} (2013), 273--309.

\bibitem{amo:jos}
E.~Artal, J.~Mart{\'i}n-Morales, and J.~Ortigas-Galindo, \emph{Intersection
  theory on abelian-quotient {$V$}-surfaces and {$\bf Q$}-resolutions}, J.
  Singul. \textbf{8} (2014), 11--30.

\bibitem{Be}
A.~Beauville, \emph{Annulation du {$H\sp 1$} pour les fibr{\'e}s en droites
  plats}, Complex algebraic varieties ({B}ayreuth, 1990), Lecture Notes in
  Math., vol. 1507, Springer, Berlin, 1992, pp.~1--15.

\bibitem{BM:15}
I.~Biswas and M.~Mj, \emph{Quasiprojective three-manifold groups and
  complexification of three-manifolds}, Int. Math. Res. Not. IMRN (2015),
  no.~20, 10041--10068.

\bibitem{brk:86}
E.V. Brieskorn and H.~Kn{\"o}rrer, \emph{Plane algebraic curves},
  Birkh{\"a}user Verlag, Basel, 1986, Translated from the German by John
  Stillwell.

\bibitem{Cohen-Suciu-boundary}
D.C. Cohen and A.I. Suciu, \emph{The boundary manifold of a complex line
  arrangement}, Groups, homotopy and configuration spaces, Geom. Topol.
  Monogr., vol.~13, Geom. Topol. Publ., Coventry, 2008, pp.~105--146.
  \MR{2508203}

\bibitem{DPS4}
A.~Dimca, {\c{S}}.~Papadima, and A.I. Suciu, \emph{Alexander polynomials:
  essential variables and multiplicities}, Int. Math. Res. Not. IMRN (2008),
  no.~3, Art. ID rnm119, 36~pp.

\bibitem{en}
D.~Eisenbud and W.D. Neumann, \emph{Three-dimensional link theory and
  invariants of plane curve singularities}, Annals of Mathematics Studies, vol.
  110, Princeton University Press, Princeton, NJ, 1985.

\bibitem{FS:14}
S.~Friedl and A.I. Suciu, \emph{K\"{a}hler groups, quasi-projective groups and
  3-manifold groups}, J. Lond. Math. Soc. (2) \textbf{89} (2014), no.~1,
  151--168.

\bibitem{gr:57}
H.~Grauert and R.~Remmert, \emph{Espaces analytiquement complets}, C. R. Acad.
  Sci. Paris \textbf{245} (1957), 882--885.

\bibitem{gr:58}
\bysame, \emph{Komplexe {R}{\"a}ume}, Math. Ann. \textbf{136} (1958), 245--318.

\bibitem{hl:cras}
H.A. Hamm and D.T. L{\^e}, \emph{Un th\'eor\`eme du type de {L}efschetz}, C. R.
  Acad. Sci. Paris S\'er. A-B \textbf{272} (1971), A946--A949.

\bibitem{eh:98}
E.~Hironaka, \emph{Plumbing graphs for normal surface-curve pairs},
  Arrangements---{T}okyo 1998, Adv. Stud. Pure Math., vol.~27, Kinokuniya,
  Tokyo, 2000, pp.~127--144.

\bibitem{kol-nem:13}
J.~Koll\'{a}r and A.~N\'{e}methi, \emph{Holomorphic arcs on singularities},
  Invent. Math. \textbf{200} (2015), no.~1, 97--147.

\bibitem{mil}
J.W. Milnor, \emph{Singular points of complex hypersurfaces}, Annals of
  Mathematics Studies, vol.~61, Princeton University Press, Princeton, N.J.,
  1968.

\bibitem{mum:61}
D.~Mumford, \emph{The topology of normal singularities of an algebraic surface
  and a criterion for simplicity}, Inst. Hautes \'Etudes Sci. Publ. Math.
  (1961), no.~9, 5--22.

\bibitem{neu:81}
W.D. Neumann, \emph{A calculus for plumbing applied to the topology of complex
  surface singularities and degenerating complex curves}, Trans. Amer. Math.
  Soc. \textbf{268} (1981), no.~2, 299--344.

\bibitem{P07}
{\c{S}}.~Papadima, \emph{Global versus local algebraic fundamental groups},
  Oberwolfach Rep. \textbf{4} (2007), no.~3, 2340--2342.

\bibitem{s:33}
H.~Seifert, \emph{Topologie {D}reidimensionaler {G}efaserter {R}\"{a}ume}, Acta
  Math. \textbf{60} (1933), no.~1, 147--238.

\bibitem{se:95}
J.-P. Serre, \emph{Rev\^etements ramifi\'es du plan projectif (d'apr\`es {S}.
  {A}bhyankar)}, S\'eminaire {B}ourbaki, {V}ol.\ 5, Soc. Math. France, Paris,
  1995, pp.~Exp.\ No.\ 204, 483--489.

\bibitem{wal:67b}
F.~Waldhausen, \emph{Eine {K}lasse von {$3$}-dimensionalen
  {M}annigfaltigkeiten. {I}, {II}}, Invent. Math. 3 (1967), 308--333; ibid.
  \textbf{4} (1967), 87--117.

\end{thebibliography}
\providecommand{\bysame}{\leavevmode\hbox to3em{\hrulefill}\thinspace}
\providecommand{\MR}{\relax\ifhmode\unskip\space\fi MR }
\providecommand{\MRhref}[2]{%
  \href{http://www.ams.org/mathscinet-getitem?mr=#1}{#2}
}
\providecommand{\href}[2]{#2}

\end{document}